\documentclass[12pt]{article}
\usepackage{fullpage}
\usepackage{mathrsfs}
\usepackage{amssymb}
\usepackage{amsmath}
\usepackage{graphicx}
\usepackage{mathtools}
\usepackage{amsmath,amsfonts}
\usepackage{todonotes}
\newtheorem{theorem}{Theorem}
\newtheorem{lemma}[theorem]{Lemma}
\newtheorem{definition}[theorem]{Definition}
\newtheorem{proposition}[theorem]{Proposition}
\newtheorem{corollary}[theorem]{Corollary}

\newenvironment{proof}{\noindent\\ \noindent\relax{\sc
     Proof}}{\samepage\par\nopagebreak\hbox     to\hsize{\hfill$\Box$
     \vspace{0mm}}}
\newcommand{\be}{\begin{equation}} \newcommand{\ee}{\end{equation}}
\newcommand{\ba}{\begin{align}} \newcommand{\ea}{\end{align}}
\newcommand{\baa}{\begin{align*}} \newcommand{\eaa}{\end{align*}}
\newcommand{\ben}{\begin{enumerate}} \newcommand{\een}{\end{enumerate}}
\newcommand{\bi}{\begin{itemize}} \newcommand{\ei}{\end{itemize}}

\newcounter{example}

\newcommand{\rd}{\mathrm{d}} 
\newcommand{\rP}{\mathrm{P}} 
\newcommand{\rE}{\mathrm{E}} 
\begin{document}

\title{General linear-fractional branching processes with discrete time} 
\author{Alexey Lindo and Serik Sagitov\\
Chalmers University of Technology and University of Gothenburg
}
\date{}
\maketitle

\begin{abstract}
We study a linear-fractional Bienaym\'e-Galton-Watson process with a general type space. The corresponding tree contour process is described by an alternating random walk with the downward jumps having a geometric distribution. This leads to the linear-fractional distribution formula for an arbitrary observation time, which allows us to establish transparent limit theorems for the subcritical, critical and supercritical cases. Our results extend recent findings for the linear-fractional branching processes with countably many types.
%
\end{abstract}

%

\section{Introduction}
Multi-type branching processes with a general  measurable space $(E,\mathcal E)$ of possible types of individuals, were addressed in monographs~\cite{Ha, Liemant1988, Sevastianov1971}, see also paper \cite{AK}. Notably, in \cite{Ja} and \cite{JN} the authors develop a full-fledged theory for the general supercritical branching processes with age dependence. These results rely upon the generalisations of the Perron-Frobenius theorem for irreducible non-negative kernels \cite{N} and Markov renewal theorems \cite{A}. Therefore, a typical limit theorem for the general branching processes involves technical conditions of irreducibility imposed on the reproduction law over the type space $E$.

This paper deals with general  Bienaym\'e-Galton-Watson processes describing branching particle systems in the discrete time setting. 
We denote by  $Z_n(A)$ the number of $n$-th generation particles whose types belong to $A\in \mathcal E$. The same generation  particles 
are assumed to produce offspring of different types independently of each other according to 
a random algorithm regulated by the parental type. A key characteristic of the multi-type reproduction law is the expectation kernel
\be\label{M}
M(x,A):=\rE_x Z_1(A),\quad x\in E,\quad A\in \mathcal E,
\ee
where the conditional expectation operator $\rE_x$
is  indexed by the type  $x$ of the ancestral particle.
The measure-valued Markov chain $\{Z_n(\rd y)\}_{n\ge0}$ has the mean value kernels
\[M^n(x,A)=\rE_x Z_n(A),\quad x\in E,\quad A\in \mathcal E,\]
computed as the powers of the kernel $M(x,dy)$
\[ M^0(x;A):=1_{\{x\in A\}},\quad M^n(x;A):=\int M^{n-1}(y,A)M(x,\rd y),\quad n\ge1.\]
Here and elsewhere the integrals are always taken over the type space $E$.

More specifically, we will study, what we call, {\it LF-processes}, branching particle systems characterised by the general linear-fractional distributions. 
At the expense of the restricted choice for the particle reproduction law, we are able to obtain explicit  Perron-Frobenius asymptotic formulas using a straightforward argument without directly referring to the general Markov chain theory. 
Our approach develops the ideas of \cite{Sa} dealing with the countably infinite type space $E$. 

%

An LF-process has a reproduction law parametrised by a triplet consisting of  a sub-stochastic kernel $K(x, dy)$, probability measure $\gamma(\rd y)$, and a number $m\in(0,\infty)$.
 Given the ancestral particle type $x$, the total offspring number $Z_1:=Z_1(E)$ is assumed to follow a linear-fractional distribution
\[\rE_xs^{Z_1}=p_0(x)+(1-p_0(x)){s\over1+m-ms},\]
With probability $p_0(x)=1-K(x,E)$ the ancestral particle has no offspring, and with probability $1-p_0(x)$, it produces a shifted-geometric number of offspring 
\[\rE_x\Big(s^{Z_1}\,|\, Z_1>0\Big)={s\over1+m-ms},\]
where parameter $m$ is independent of $x$.

Given that $Z_1=k$ and $k\ge1$, one of the $k$ offspring will be distinguished and called a marked offspring. We will assume that the assignment of types to $k$ offspring is done independently using the probability distribution $\kappa_x(\rd y):=K(x, \rd y)/K(x,E)$ for the marked offspring, and distribution  $\gamma(\rd y)$ for the remaining $k-1$ offspring. Observe that only the type of the marked offspring is allowed to depend on the parental type $x$.
%
%
%

The above mentioned assumptions result in an important feature of  LF-processes. The kernel \eqref{M} of an LF-process has a particular structure 
\begin{equation}\label{mrg}
  M (x,A)= K(x,A) + K(x,E)m \gamma(A),
\end{equation}
where the first term, $K(x,A)$, is the contribution of the marked offspring and the second term is the joint contribution of other offspring. 
To summarise, the framework of LF-processes has a reasonable level of generality: it is broad enough to contain a variety of interesting examples, yet restrictive enough to allow for transparent limit theorems shedding light onto the general theory of multi-type branching processes.

Our presentation in Section \ref{SGLF} starts with a  more formal definition of general linear-fractional distributions. It is shown that for an arbitrary generation $n$ of the LF-process, the random measure $Z_n(\rd y)$  has a linear-fractional distribution, see Theorem \ref{P2}.
In Section \ref{Spf} we obtain a transparent form of the Perron-Frobenius theorem (Theorem \ref{Rec}) for the powers of the kernel \eqref{mrg} using a generating function approach adapted from \cite{Sa}.

The linear-fractional property stated in  Theorem \ref{P2} is proven in terms of 
an inherent linear-fractional Crump-Mode-Jagers (CMJ) process, described in Section \ref{Scmj}. An alternative picture of the LF-process as a branching random walk over the state space $E$ provides with new insight into our model. Namely, one can think of CMJ-individuals walking over $E$ according to the Markov transition rules with kernel $K(x,dy)$. Each individual alive at the current moment, produces a geometric number of offspring with mean $m$. All newborn individuals have independent starting positions with
 the common distribution $\gamma(\rd y)$.

In Section \ref{Scl} we demonstrate how the  simple
conditions and clear statements of the Perron-Frobenius Theorem \ref{Rec} relate to less transparent general results of this kind (summarised, for example, in \cite{MT} and \cite{N}).

Finally, Section \ref{Sf} presents three basic limit theorems for the subcritical, critical, and supercritical LF-processes. The obtained asymptotic formulas are clearly expressed in terms of the defining triplet $\{K(x,\rd y),\gamma(\rd y),m\}$.

Our findings are illustrated by a special family of LF-processes with the type space $E=(0,\infty)$ whose reproduction law is characterised by three positive parameters 
$(\lambda, \mu,m)$. 

\section{General linear-fractional distributions}\label{SGLF}

\begin{definition}\label{Dlf}
An integer-valued random measure $Z(\rd y)$ on $(E,\mathcal E)$ with total mass $Z:=Z(E)$ is said to have a linear-fractional distribution 
if for some $p_0\in[0,1)$ and $m_0\in(0,\infty)$,
\begin{align*}
& \rP(Z= 0) = p_0,\\
&\rP(Z=k|Z> 0)  =  {m^{k-1}\over (1+ m)^k},\quad k\ge1,
\end{align*}
and conditionally on $Z=k$, 
$$Z(A)\stackrel{d}{=}1_{\{X_1\in A\}}+\ldots+1_{\{X_k\in A\}},\quad A\in\mathcal E,$$
where $X_1,X_2,\ldots$ are  independent random points on $E$ with 
$$\rP(X_1\in A)=\kappa(A),\quad \rP(X_i\in A)=\gamma(A),\ \ i\ge 2,$$
for two given probability measures $\kappa(\rd y)$ and $\gamma(\rd y)$.
\end{definition}

The above defined  linear-fractional distribution of  $Z(\rd y)$ is compactly described by its generating functional 
\[
  \rE \exp\Big\{\int Z(\rd y) \ln h(y)\Big\} 
  = p_0+ (1-p_0){\int h(y)\kappa(\rd y) \over 1 + m -  m\int h(y)\gamma(\rd y)}
\]
having, as indicated by the distribution name, a linear-fractional form.


\begin{theorem} \label{P2}
Consider an LF-process $\{Z_n(\rd y)\}$ defined by a  triplet $\{K(x,\rd y),\gamma(\rd y),m\}$ so that 
\be\label{reLa}
  \rE_x \exp\Big\{\int Z_1(\rd y) \ln h(y)\Big\}
  = 1 - K(x,E)+ {\int h(y)K(x,dy) \over 1 + m -  m\int h(y)\gamma(\rd y)}.
\ee
Then for each $n\ge2$,
\be\label{reLan}
 \rE_x \exp\Big\{\int Z_n(\rd y) \ln h(y)\Big\}
 = 1 - K_n(x,E)+ {\int h(y)K_n(x,dy) \over 1 + m_n -  m_n\int h(y)\gamma_n(\rd y)},
\ee
where the  triplet $\{K_n(x,\rd y),\gamma_n(\rd y),m_n\}$ is uniquely specified by the relations
\begin{align}
  m_{n} &= m \sum_{k = 0}^{n - 1} \int M^{k}(x,E)\gamma(\rd x), \label{mn}\\
 \gamma_{n}(A) &=m_{n}^{-1 } m\sum_{k = 0}^{n - 1} \int M^{k}(x,A)\gamma(\rd x), \label{tn}\\
  K_{n}(x,A) &= M^{n}(x,A) - {m_{n} \over 1 + m_{n}} M^{n}(x,E) \gamma_{n}(A). \label{rjn}
\end{align}
\end{theorem}

The proof of Theorem \ref{P2} is essentially the same as that of Theorem 3 in~\cite{Sa}. 
 A sketch of the proof is postponed until the end of Section \ref{Scmj}.

\begin{corollary}
Consider the LF-process from Theorem \ref{P2} and put $Z_n=Z_n(E) $. Then the  survival probability satisfies  the following transparent formula 
\begin{equation}\label{nex}
  \rP_{x}(Z_{n}> 0)  =  (1 + m_{n})^{-1} M^{n}(x,E).
\end{equation}
Furthermore, conditionally on the ancestral type $x$ and the survival event $\{Z_{n}> 0\}$, we have  
\be\label{tex} 
 \rE_x \Big(\exp\Big\{\int Z_n(\rd y) \ln h(y)\Big\}
|Z_n>0\Big)=K_n(x,E)^{-1}{\int h(y)K_n(x,dy) \over 1 + m_n -  m_n\int h(y)\gamma_n(\rd y)}.
\ee

\end{corollary}
\begin{proof}
From \eqref{reLan}, we find
\[  \rP_{x}(Z_{n}> 0)  = K_{n}(x,E), \]
which together with  \eqref{rjn} yields the stated formula for the survival probability.
The second claim is another consequence of  the formula \eqref{reLan}.
\end{proof}

\begin{figure}
\centering
 \includegraphics[height=9cm]{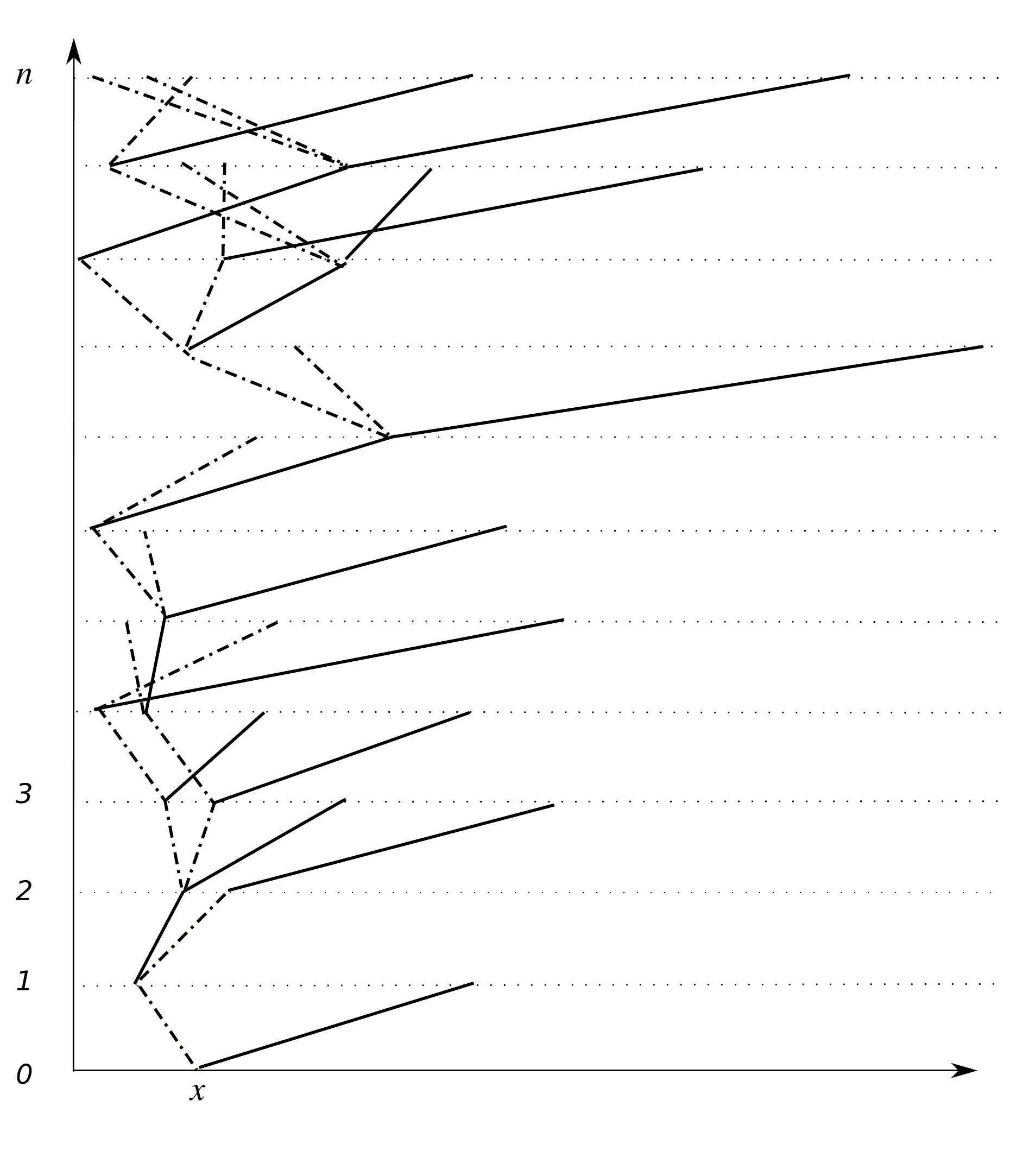}  
 \caption{ a tree (truncated at height $n$) illustrating a possible development of a $(\lambda,\mu,m)$-branching process with a small $\lambda$ (relative to $\mu$). The solid lines connect mother-particles to their marked daughters. The horizontal axis represents the type space $E=(0,\infty)$. Particles located further to the right have exponentially smaller survival probability.  
}
\label{F1}
\end{figure}

\noindent{\bf Example}: $(\lambda, \mu, m)$-branching process.

\noindent Take $E=(0,\infty)$, and consider three positive parameters
$\lambda>0$, $\mu>0$, $m>0$.
Let $Y_u$ stand for a random variable with an exponential distribution with parameter $u$, and put
$$K(x,A)=e^{-x}\rP(x+Y_\lambda\in A),\qquad \gamma(A)=\rP(Y_\mu\in A).$$
The corresponding LF-process will be called a $(\lambda, \mu, m)$-branching process. With this particular choice of the defining triplet, a particle of type $x>0$ has at least one offspring with probability $1-p_0(x)=e^{-x}$. 
For this example, we obtain
\begin{align*}
 K^n(x,A)&=e^{-x}\rE\big(e^{-(x+Y_\lambda^{(1)})}\cdots e^{-(x+Y_\lambda^{(1)}+\ldots+Y_\lambda^{(n-1)})}1_{\{x+Y_\lambda^{(1)}+\ldots+Y_\lambda^{(n)}\in A\}}\big)\\
&=e^{-nx}\rE\big(e^{-(n-1)Y_\lambda^{(1)}}e^{-(n-2)Y_\lambda^{(2)}}\cdots e^{-Y_\lambda^{(n-1)}}1_{\{x+Y_\lambda^{(1)}+\ldots+Y_\lambda^{(n)}\in A\}}\big),
\end{align*}
where $Y_\lambda^{(1)},Y_\lambda^{(2)},\ldots$ are independent exponentials with parameter $\lambda$. Putting $A=E$, we find
\begin{align*}
 K^n(x,E)&=e^{-nx}\rE e^{-(n-1)Y_\lambda}\rE e^{-(n-2)Y_\lambda}\cdots \rE e^{-Y_\lambda}\\
 &={\lambda^{n-1}e^{-nx}\over (\lambda+n-1)(\lambda+n-2)\cdots (\lambda+1)}={\lambda^ne^{-nx}\Gamma(\lambda)\over  \Gamma(\lambda+n)}.
\end{align*}

Figure \ref{F1} illustrates a possible shape of a tree generated by an LF-process  with  $\lambda$ being much smaller than $\mu$.


\section{Perron-Frobenius theorem }\label{Spf}
According to Theorem \ref{P2}, the asymptotic behaviour of the LF-process is fully determined by the asymptotic behaviour of  $M^n(x,A)$ as $n\to\infty$, which is the subject of the this section.
Here we analyse the growth rate of $M^n(x,A)$  in terms of the generating functions
\[M^{(s)}(x,A)=\sum_{n=0}^\infty s^nM^n(x,A),\quad K^{(s)}(x,A)=\sum_{n=0}^\infty s^nK^n(x,A),\]
where the kernel powers $K^n(x,A)$ are defined similarly to $M^n(x,A)$. Warning: $K^n(x,A)$  should not be confused with $K_n(x,A)$ introduced in Theorem \ref{P2}.

A key tool in our analysis is the  generating function
\begin{equation} \label{qn}
f(s)= \sum_{n \ge 1} d_{n}s^{n},\quad d_n=\int K^n(x,E)\gamma(\rd x),\quad n\ge1,
\end{equation}
whose radius of convergence will be denoted by
\[ R_{*} := \inf \{ s > 0 \colon f(s) = \infty \}.\]
From \eqref{qn}, we see that
\begin{equation}\label{KsE}
 \int K^{(s)}(x,E)\gamma(\rd x)=1+f(s).
\end{equation}
Therefore,  if $f(s)<\infty$, then $\gamma(E_{s})=1$, 
where
 \[E_{s}=\{x\in E: K^{(s)}(x,E)<\infty\},\quad s\in(0,\infty).\]

%
%
%
%
%
\begin{theorem}\label{Rec}
Suppose that $f(R_*) \ge 1 / m$, so that there is a unique $R\in(0,\infty)$ satisfying $mf(R)=1$. Then
\begin{align*}
 u(x)&=(1+m)(K^{(R)}(x,E)-1)1_{\{x\in E_R\}},\\
 \nu(A)&={m\over1+m}\int K^{(R)}(y,A)\gamma(\rd y),
\end{align*}
are well-defined and satisfy 
\[\int u(x)\gamma(\rd x)={1+m\over m},\quad \nu(E)=1,\quad \int u(x)\nu(\rd x)=mRf'(R).\]
Put $\rho=R^{-1}$. For $x\in E_R$, we have
\[\int u(y)M(x,dy)=\rho u(x),\quad \int M(y,A)\nu(\rd y)=\rho\nu(A).\]
Moreover, if $f'(R)<\infty$, then
 \[R^{n}M^n(x,A)\to 
  {u(x)\nu(A)\over mRf'(R)},\quad n\to\infty,
\]
and if $f'(R)=\infty$, then  $R^{n}M^n(x,A)\to0$.
\end{theorem}

\begin{proof} All parts of this statement, except the last one, are checked by straightforward calculations. In particular, using \eqref{KsE}, we obtain
\[\int u(x)\gamma(\rd x)=(1+m)f(R)={1+m\over m},\]
and also
\begin{align*}
 \int u(x)\nu(\rd x)&=m\sum_{n=1}^\infty R^n\int\int K^n(x,E)K^{(R)}(y,dx)\gamma(\rd y)\\
 &= m\sum_{k=1}^\infty kd_kR^k= mRf'(R).
\end{align*}

To prove the last part we show first that
\begin{equation}\label{M^{(s)}}
 M^{(s)}(x,A)=K^{(s)}(x,A)+(K^{(s)}(x,E)-1)m\int M^{(s)}(y,A)\gamma(\rd y).
\end{equation}
Indeed, using \eqref{mrg}, we find
\[M^n(x,A)=\int M^{n-1}(y,A)K(x,dy)+mK(x,E)\int M^{n-1}(y,A)\gamma(\rd y).\]
Reiterating this relation we obtain
\[M^n(x,A)=K^{n}(x,A)+m\sum_{i=0}^nK^i(x,E)\int M^{n-i}(y,A)\gamma(\rd y)-m\int M^{n}(y,A)\gamma(\rd y),\]
which leads to \eqref{M^{(s)}} as we go from the sequences to their generating functions.

By iterating \eqref{M^{(s)}} once, we obtain
\begin{align*}
 M^{(s)}(x,A)&=K^{(s)}(x,A)+(K^{(s)}(x,E)-1)m\int K^{(s)}(y,A)\gamma(\rd y)\\
 &\quad +(K^{(s)}(x,E)-1)m^2f(s)\int M^{(s)}(y,A)\gamma(\rd y).
\end{align*}
Assuming $mf(s)<1$ and reiterating we obtain
\begin{align*}
 M^{(s)}(x,A)&=K^{(s)}(x,A)+{m\over 1-mf(s)}(K^{(s)}(x,E)-1)\int K^{(s)}(y,A)\gamma(\rd y).
\end{align*}
If we now apply Lemma \ref{Fe} below with
\[a(s)=mf(sR),\quad b(s)=m(K^{(sR)}(x,E)-1)\int K^{(sR)}(y,A)\gamma(\rd y),\]
then using $a'(1)=mRf'(R)$, we derive
 \[R^{n}(M^n(x,A)-K^n(x,A))\to 
  \left\{
\begin{array}{ll}
 \dfrac {u(x)\nu(A)}{mRf'(R)} &    \mbox{if } f'(R)<\infty,\\
0  &   \mbox{if }  f'(R)=\infty.
\end{array}
\right.
\]
It remains to observe that $R^{n}K^n(x,A)\to0$ as $n\to\infty$ for all $x\in E_R$ and $A\in\mathcal E$.
\end{proof}

The next lemma is the renewal theorem from \cite[Chapter XIII.10]{Fe1}.
\begin{lemma} \label{Fe}
Let $a(s)=\sum_{n=0}^\infty a_ns^n$ be a probability generating function and $b(s)=\sum_{n=0}^\infty b_ns^n$ is a generating function for a non-negative sequence, so that $a(1)=1$ while $b(1)\in(0,\infty)$. Then the  non-negative sequence $\{c_n\}$ defined by $\sum_{n=0}^\infty c_ns^n={b(s)\over1-a(s)}$ is  such that $c_n\to {b(1)\over a'(1)}$ as $n\to\infty$.
\end{lemma}

\section{Embedded Crump-Mode-Jagers process}\label{Scmj}

A crucial feature of the LF-process is the existence of an embedded linear-fractional Crump-Mode-Jagers (CMJ) process described in this section. We show that all the key entities involved in the Perron-Frobenius Theorem \ref{Rec} 
have a transparent probabilistic meaning in terms of this CMJ-process. 
Recall that a single-type CMJ-process models an asexual population with overlapping generations, see \cite{JS}. The CMJ-model is described in terms of individuals rather than particles, since a CMJ-population is set out in the real-time framework, in contrast to the  generation-wise setting for the Bienaym\'e-Galton-Watson process, cf \cite[Section 3]{Sa}.

With a given triplet $\{K(x,\rd y),\gamma(\rd y),m\}$  consider  the LF-process stemming from a particle whose type is randomly chosen using the distribution $\gamma(\rd x)$. The embedded CMJ-process stems from an ancestral {\it individual }whose life history represents the evolution of the lineage of the marked descendants of the ancestral {\it particle} $\pi_0$. Consider the sequence of descendants of  $\pi_0$ consisting of its marked child $\pi_1$, the marked child $\pi_2$ of the $\pi_1$, the marked child $\pi_3$ of the marked grandchild P2, and so on until this lineage halts by a particle $\pi_L$ having no children. 
It turns out that the life length $L$ of the ancestral individual has  the tail probabilities $\rP(L>n)=d_n$ with the tail generating function \eqref{qn}.
In particular, 
$$\rP(L\ge1)=1,\quad  \rE L=1+f(1).$$

Each unmarked daughter particle produced by any of the particles in the lineage $(\pi_0,\ldots,\pi_{L-1})$ will be treated as the originator of a new individual which is considered to be a  daughter of the ancestral individual. As a result, the ancestral individual produces random numbers of offspring at times $1,\ldots,L-1$. The corresponding litter sizes are mutually independent and have the same geometric distribution with mean $m$. (For a continuous time version of the linear-fractional CMJ-processes, see \cite{La}.) The newborn individuals live independently according to the same life law as their ancestor. Thus defined CMJ-process has the population size at time $n$ coinciding with generation size $Z_n=Z_n(E)$ of the LF-process having parameters $\{K(x,\rd y),\gamma(\rd y),m\}$ and starting from a particle whose type has distribution $\gamma(\rd x)$. See Figure \ref{F} for illustration.

\begin{figure}
\centering
 \includegraphics[height=9cm]{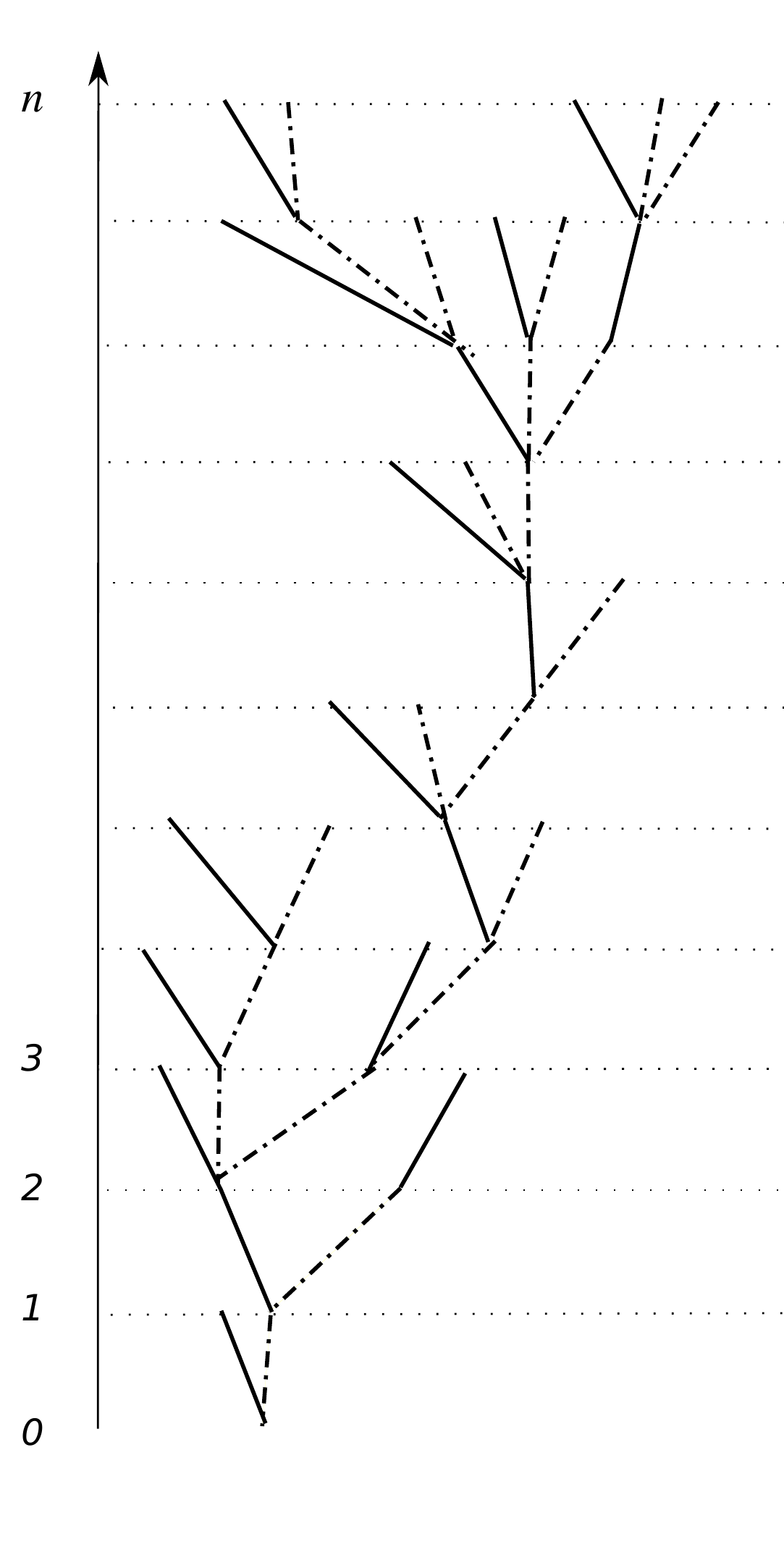}\qquad\qquad 
 \includegraphics[height=9cm]{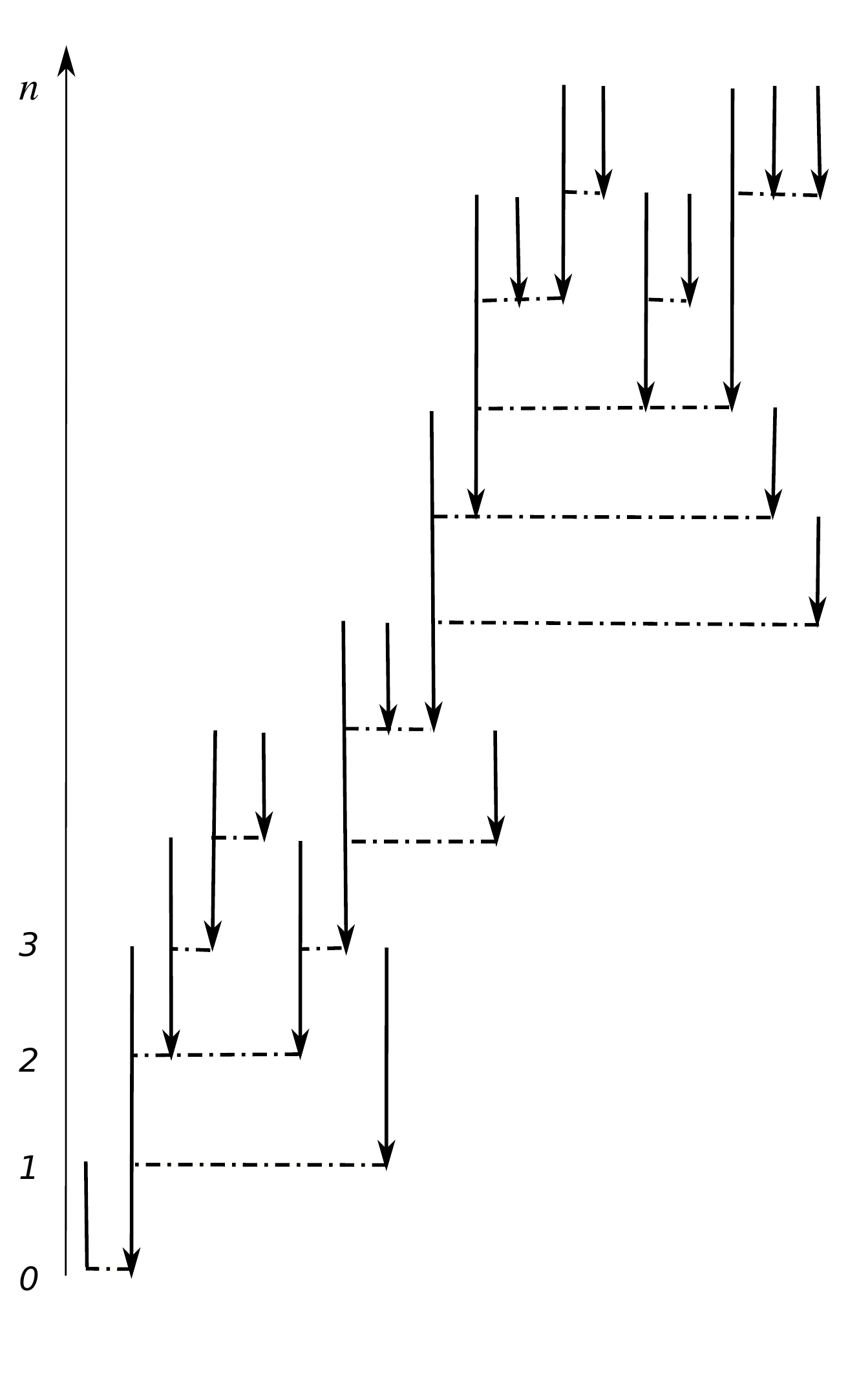}
\caption{ alternative drawings of the tree from Figure \ref{F1}  when particle types are removed. {\it Left panel}. 
The branches connecting mother-particles to their offspring are drawn such that the marked branch is always the leftmost among the sibling branches.
{\it Right panel}. The  vertical lines represent CMJ-individuals. The arrows say that individuals born at time $k$ are not added to the population size $Z_k$.
}
\label{F}
\end{figure}

The embedded single-type CMJ-process conceals the information on the types of the particles. To recover this information we introduce additional labelling of individuals using the types of underlying particles. The individual's label evolution over the type space $E$ can be described by a Markov chain whose state space  $\mathring{E} = E \cup \Delta$ is the particle type space $E$ augmented with a graveyard state. The transition probabilities of such a chain are given by a stochastic kernel $\mathring{K}$ 
\[
\mathring{K}(x, A) = \left\{
\begin{array}{lll}
K(x,A)  &  \mbox{for }  x\in E, &   A\in\mathcal E,  \\
 1 - K(x,E) &  \mbox{for }  x\in E,&   A=\{\Delta\},   \\
 1 &   \mbox{for }    x=\Delta,& A=\{\Delta\}.
\end{array}
\right.
\]
In terms of this Markov chain, the life length $L$ is the time until absorption at the graveyard state, provided the initial state distribution is $\gamma(\rd x)$.

Turning to Theorem \ref{Rec}, we see that the Perron-Frobenius root $\rho$ satisfies equation $f(1/\rho)=1/m$. This equation puts together two ingredients $(f,m)$ of the CMJ-individual reproduction law: its life-length tail generating function and the mean offspring number per unit of time. 

In the age-dependent setting the population growth rate is not given by the mean value $\mu=mf(1)$ for the number of offspring produced by a CMJ-individual during its whole life.
 The growth rate is described by the so called  Malthusian parameter $\alpha$.  
Using (7) from \cite{JS} one can compute   $\alpha$ from the equation 
 $\sum_{n=1}^\infty e^{-\alpha n}a_n=1,$
 where $a_n$ stands for the mean number of offspring produced by the ancestral individual at time $n$. 
Since $a_n=md_n$, we conclude that $mf(e^{-\alpha})=1$, and provided the Malthusian parameter exists, we have
\[ \alpha=\ln \rho=-\ln R. \]
For a given $\alpha$, the mean age at childbearing, see \cite{J} and \cite{JS}, corresponding to the average generation length, is computed as
\[
  \beta=m \sum_{n=1}^\infty n d_n e^{-\alpha n}=mRf'(R),
\]
so that $\beta$ is either finite or infinite depending on whether $f'(R)$ is finite or not. \\

\noindent{\bf Example}: $(\lambda, \mu, m)$-branching process.\\
For the 3-parameter LF-process introduced in Section \ref{SGLF}, we find 
$$\rP(L>n)=\int_0^\infty {\Gamma(\lambda)\lambda^ne^{-nx}\over  \Gamma(\lambda+n)}\mu e^{-\mu x}dx={\Gamma(\lambda)\lambda^n\mu\over  \Gamma(\lambda+n)(\mu+n)}.$$
From 
$$\rP(L>n)\sim {\Gamma(\lambda)\lambda^n\mu\over  n^{\lambda}n!},\quad n\to\infty,$$ 
we see that all moments of $L$ are finite. 
Observe that \eqref{qn} can be written as
\[f(s)=\Phi(\lambda s)-1,\qquad \Phi(s)=\sum_{n=0}^\infty {\mu \Gamma(\lambda)s^n\over  \Gamma(\lambda+n)(\mu+n)},\]
where $\Phi(s)=\sum_{n=0}^\infty \phi_ns^n$ is a generalised hypergeometric function in view of
$$
{\phi_{n+1}\over\phi_n}={\mu+n\over  (\lambda+n)(\mu+n+1)}={(1+n)(\mu+n)\over  (\lambda+n)(\mu+1+n)(1+n)}.$$
We can write $\Phi(s)= {}_2F_2(1, \mu; \lambda, \mu+1;s)$ using the standard notation for  generalised hypergeometric functions. 
Figure \ref{F2} 
 illustrates the complex relationship between the expected life length of CMJ-individuals and the average age at child-bearing.\\
\begin{figure}
\centering
 \includegraphics[width=7cm]{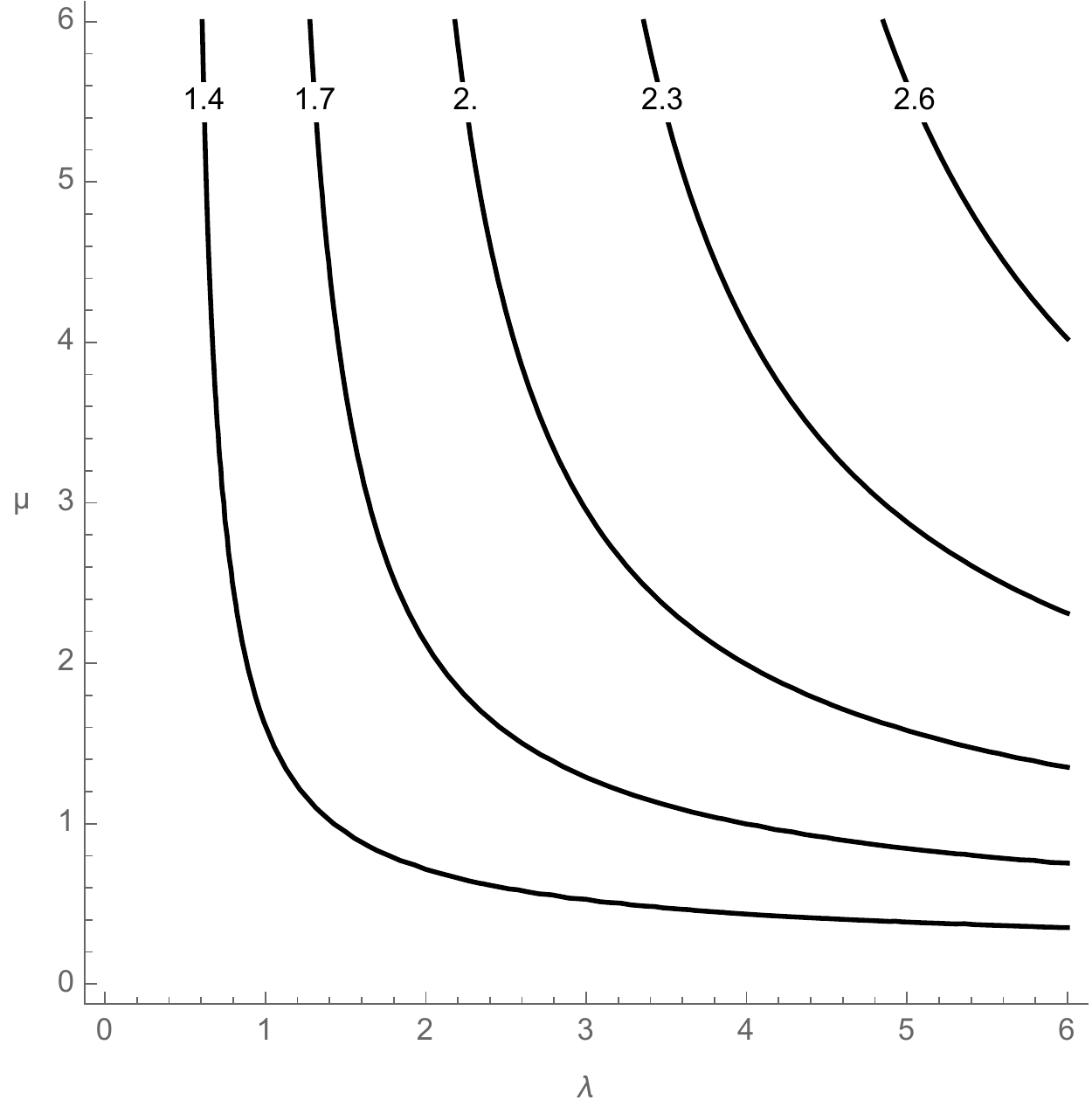}\qquad 
  \includegraphics[width=7cm]{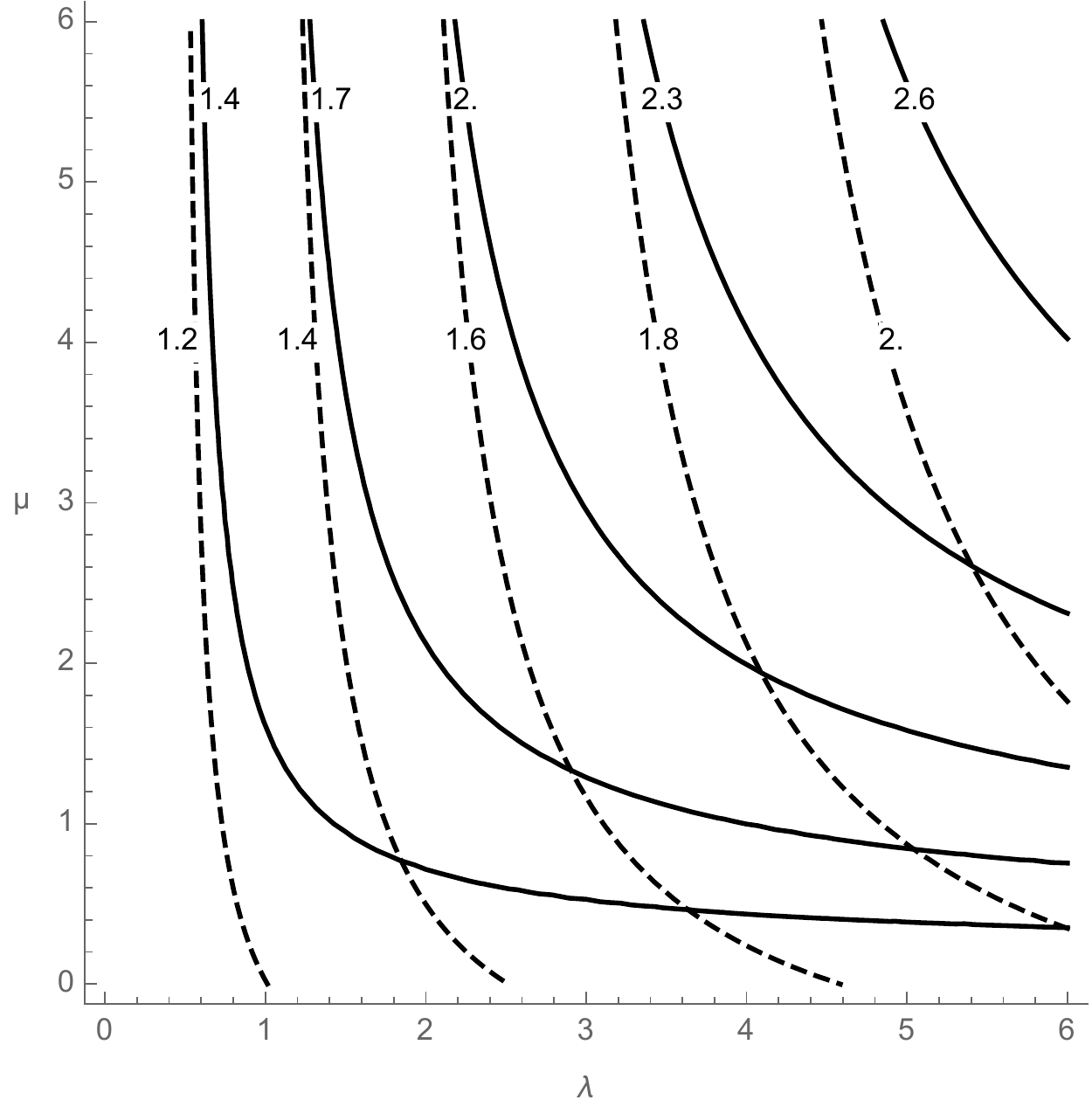}
  \vspace{0.5cm}
  
 \includegraphics[width=7cm]{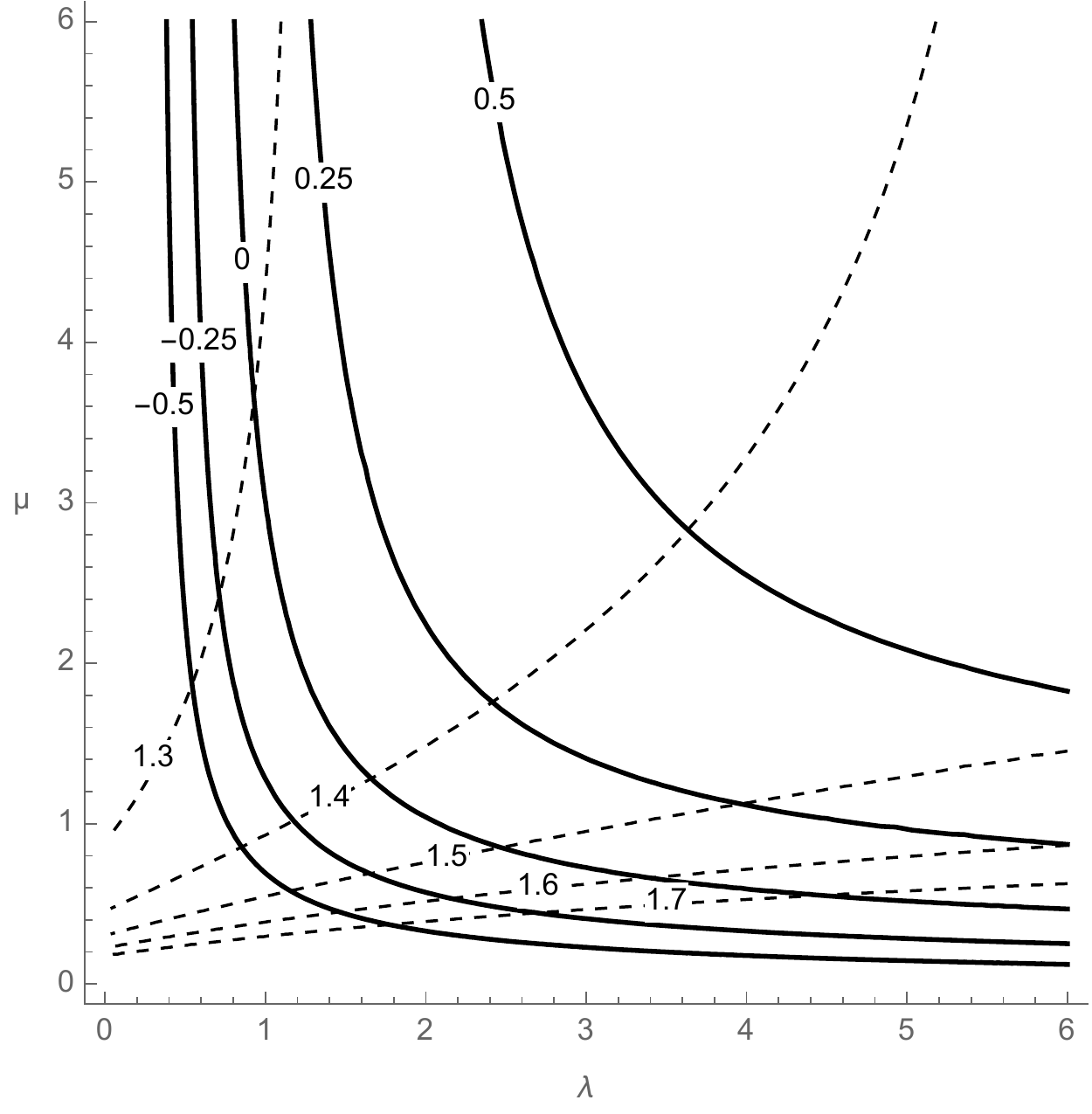}
 \qquad 
 \includegraphics[width=7cm]{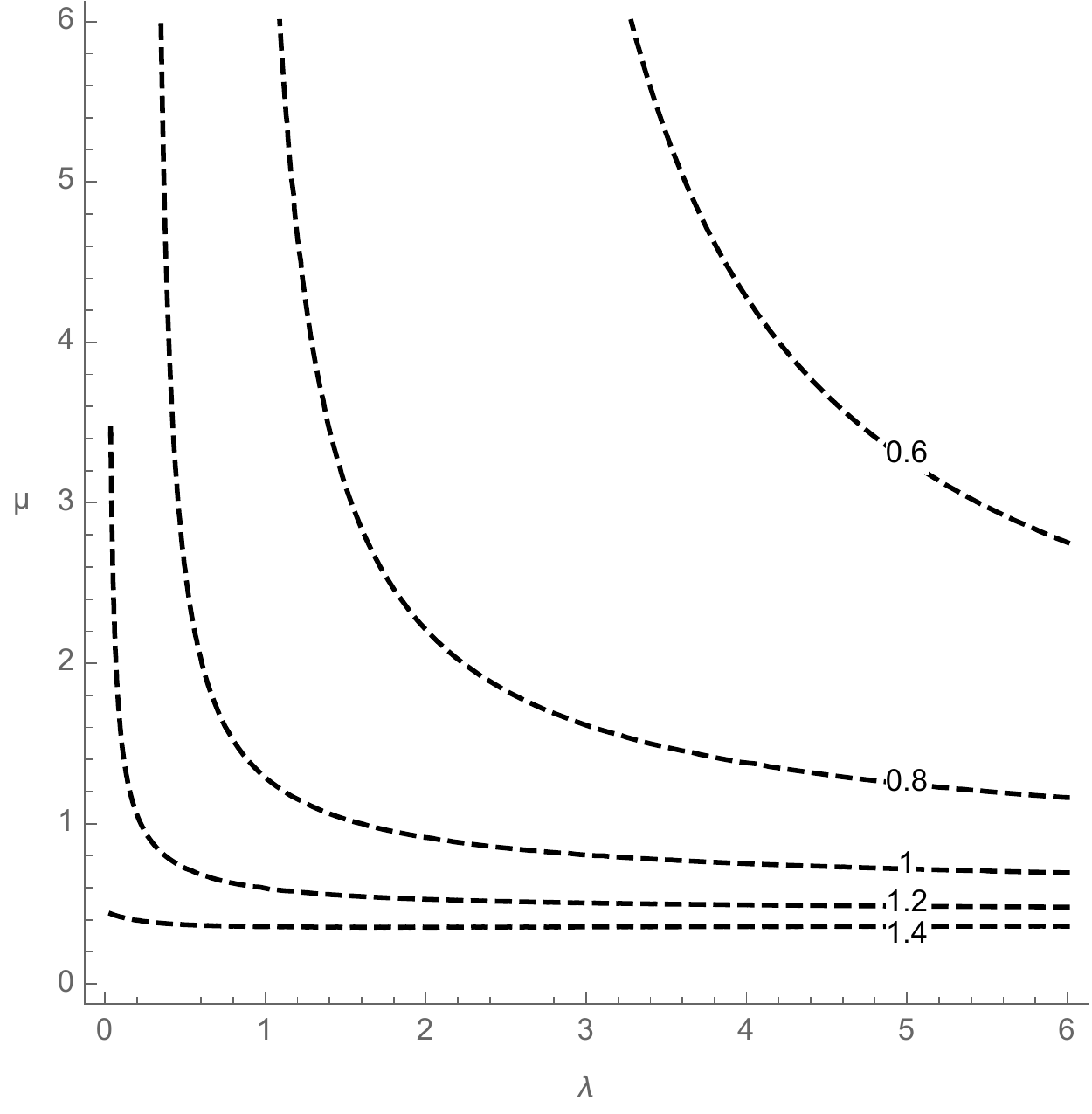}

\caption{
{\it The first row of panels}. 
Left: the expected life length $\rE L$ as a function of $(\lambda,\mu)$. 
Right: we fix $\alpha=0$ and compare the average life length $\rE L$ to the average age at child-bearing $\beta$ (dashed lines).
{\it The second row of panels}. Take $m=2$. Left:  the solid lines give the values of the Malthusian parameter $\alpha$, the dashed lines describe $\beta$. Right: the dashed lines give the ratio $\beta/\rE L$.}
\label{F2}
\end{figure}

{\sc Sketch of a proof of Theorem \ref{P2}}.
The embedded CMJ-process generates random trees (see the right panel of Figure \ref{F1}) having tree contours of simple structure. Going around a  CMJ-tree, one performs an  alternating random walk, where an instantaneous upward jump, whose size is distributed as the life length $L$, is followed by a geometric number of downward unit-jumps until one hits the nearest branch turning the random walk upwards, see \cite[Section 3]{Sa} for details. 

In terms of such a contour process around the tree truncated at the observation level $n$, the LF-process size at time $n$ is given by the number of the alternating random walk excursions on the level $n$. The number of such excursions, once the level $n$ is reached, is geometric and independent of the ancestral type. This implies the linear-fractional form of the distribution for the random measure $Z_n(\rd y)$, see \cite[Section 4]{Sa} for details. 
 
Having established the stated linear-fractional distribution property by the contour process argument, we have to verify that relations  \eqref{mn} - \eqref{rjn} indeed specify the triplet defining the $n$-th generation distribution. The key relation \eqref{tn} relies on the spinal representation trick explained in Section 7.1 of \cite{Sa}. The expression  \eqref{mn} for $m_n$ is a straightforward corollary of \eqref{tn}, while  \eqref{rjn} is obtained from the following analog of  \eqref{mrg} 
 \[ M^n (x,A)= K_n(x,A) + K_n(x,E)m_n \gamma_n(A),\]
 according to which
 \[ M^n (x,E)= K_n(x,E)(1+m_n).\]

\section{Positive recurrence over the type space}\label{Scl}

Assume that $ f'(R)<\infty$ or equivalently $\beta<\infty$. According to Theorem \ref{Rec},  we can distinguish among three major reproduction regimes. We will call the  LF-process 
\begin{itemize}
\item {\it subcritical} if $mf(1)<1$, equivalently $\rho<1$, $R>1$, $\alpha<0$, in this case the expected generation size decreases exponentially as $\rho^n$;
\item  {\it critical} if $mf(1)=1$, equivalently $\rho=1$, $R=1$, $\alpha=0$,  in this case  the expected generation size measure stabilises;
\item  {\it supercritical} if $mf(1)>1$, equivalently $\rho>1$, $R<1$, $\alpha>0$,  in this case  the expected generation size increases exponentially as $\rho^n$.
\end{itemize}

Recall that $R$ is defined by $mf(R)=1$ provided $f(R_*)\ge1/m$. We extend this definition of $R$ by putting $R=R_*$ in the case $f(R_*)<1/m$. The next lemma gives another perspective at the meaning of parameter $R$.

\begin{lemma} \label{Lr}
The power series $M^{(s)}(x,A)$ have the same radius of convergence $R$ for all $A\in\mathcal E$ and for $\gamma$-almost every $x\in E$.
\end{lemma}
\begin{proof}
 Integrating \eqref{M^{(s)}}  with respect to measure $\gamma$ and using \eqref{KsE}, we find
\[\int M^{(s)}(x,A)\gamma(\rd x)=\int K^{(s)}(x,A)\gamma(\rd x)+mf(s)\int M^{(s)}(y,A)\gamma(\rd y).
\]
It follows that if   $mf(s)<1$, then
\begin{equation}\label{rad}
 \int M^{(s)}(x,A)\gamma(\rd x)=
{\int K^{(s)}(x,A)\gamma(\rd x)\over 1-mf(s)} ,  
\end{equation}
implying
$\int M^{(s)}(x,A)\gamma(\rd x)<\infty$. On the other hand, if $mf(s)\ge1$, the last integral is infinite. 
 Turning to the definition of $R$, we conclude that the statement is true.
\end{proof}

Using the terminology of the theory of general Markov chains and irreducible kernels \cite[Ch 3.3]{N}, our Lemma \ref{Lr} says that $R$ is the convergence parameter of the kernel $M(x,dy)$. Furthermore, 
we see that  if $f(R_*) \ge 1 / m$, then $M(x,dy)$ is $R$-recurrent, while if $f(R_*) <1 / m$, then $M(x,dy)$ is $R$-transient.

The  
$R$-recurrent case is further split in two sub-cases. 
According to Theorem \ref{Rec}, the $R$-recurrent 
kernel $M(x,dy)$  is $R$-null recurrent if $f'(R)=\infty$, and $R$-positive recurrent if $f'(R)<\infty$, cf \cite[Ch 5]{N}. Observe also that the pair $(u(x),\nu(\rd y))$ introduced in  Theorem \ref{Rec} are  invariant (function, measure) for the kernel $M(x,dy)$.

The theory of irreducible kernels is built around the so-called minorisation condition. It turns out that the key relation for our model \eqref{mrg} automatically produces a relevant minimisation condition
\[M(x,A)\ge mK(x,E)\gamma(A).\]
In this context the pair $(K(x,E), \gamma(\rd y))$ is called an atom for the kernel $M(x,dy)$. The existence of an atom allows to construct an embedded renewal process \cite[Ch 4]{N} and carry over most of the results from the theory of countable matrices to the general state space. The approach of this paper allows for the kernels satisfying \eqref{mrg} to circumvent the use of  such general theory for obtaining the Perron-Frobenius theorem.

Observe that the $(\lambda,\mu,m)$-branching processes are positively recurrent LF-processes. 
 In the next section we 
derive three limit theorems for the LF-processes in the $R$-positively recurrent case.

\section{Basic limit theorems for the LF-processes}\label{Sf}
Combining Theorems \ref{P2} and \ref{Rec} we establish next  three propositions stated under a common assumption
\be\label{cond}
R\in(0,\infty),\quad x\in E_R,\quad  f'(R)<\infty.\ee
These propositions are basic asymptotic results 
for the general LF-processes extending similar statements for the countably infinite $E$ in \cite[Section 6]{Sa}.

\begin{proposition}\label{p1} Assume \eqref{cond} and let  $\rho<1$. Then as $n\to\infty$,
\be
 \rP_{x}(Z_{n} > 0) \sim{1-mf(1)\over(1+m)\beta}\rho^nu(x).
\label{epr}
\ee
Furthermore, conditionally on the ancestral type $x$ and the survival event $\{Z_{n} > 0\}$, 
\[ \rE_x \Big(\exp\Big\{\int Z_n(\rd y) \ln h(y)\Big\}
|Z_n>0\Big)\to {\int h(y)\tilde\kappa_x(dy) \over 1 + \tilde m -  \tilde m\int h(y)\tilde \gamma_(\rd y)},
\]
 as $n\to\infty$, where 
\begin{align*}
\tilde m&={m(1+f(1))\over 1-mf(1)},\quad \tilde\gamma(A)={1\over 1+f(1)}\int K^{(1)}(x,A)\gamma(\rd x),\\
\tilde\kappa_x(A)&={m\over 1-mf(1)}\int \Big\{K^{(R)}(x,A)-K^{(1)}(x,A)\Big\}\gamma(\rd x).
\end{align*}
\end{proposition}
\begin{proof}  From \eqref{mn} and  \eqref{rad} we obtain 
$$m_n\to {m(1+f(1))\over1-mf(1)},$$ 
which together with  \eqref{nex} implies \eqref{epr}. 
The stated convergence of the conditional distribution of $Z_n(\rd y)$ follows from \eqref{tex}. Indeed,
by \eqref{tn}, we have
 \[\gamma_{n}(A) \to {m\over \tilde m}  \int M^{(1)}(x,A)\gamma(\rd x)={\int K^{(1)}(x,A)\gamma(\rd x)\over 1+f(1)}. \]
On the other hand, using  \eqref{rjn}, we find
\[  K_{n}(x,A) \sim\rho^n{u(x)\over \beta}\Big(\nu(A)-{m(1+f(1))\tilde\gamma(A)\over 1+m}\Big), \]
yielding
\[  {K_{n}(x,A) \over K_{n}(x,E)} \to{m\over 1-mf(1)}\int \Big\{K^{(R)}(x,A)-K^{(1)}(x,A)\Big\}\gamma(\rd x). \]
\end{proof}

\begin{proposition}\label{p2} Assume \eqref{cond} and let  $\rho=1$. Then as $n\to\infty$,
\[\rP_{x}(Z_{n} > 0) \sim \beta n^{-1}(1+m)^{-1} u(x).\]
Moreover, for any measurable probe function $w$ with 
$\int w(y)\nu(\rd y)\in(0,\infty),$
and for any $x\ge0$, 
\begin{align*}
\rP_{x}\Big({\int w(y)Z_n(\rd y)\over \int w(y)\nu(\rd y)}>nx | Z_{n} > 0\Big)\to e^{-x\beta/(1+m)}.
\end{align*}
In other words, conditionally  on non-extinction the scaled random measure $n^{-1}Z_n(\rd y)$ weakly converges to $X\nu(\rd y)$, where  $X$ is exponentially distributed with mean $1+m$. 
\end{proposition}
\begin{proof} 
Lemma \ref{Fe} and relations  \eqref{mn}, \eqref{rad} imply that in the critical case  
\[
m_n\sim n(1+m)\beta^{-1}.
\]
Moreover, by \eqref{tn} and \eqref{rjn} we get
 \begin{align*}
 \gamma_{n}(A) &= {m\over m_{n}} \sum_{k = 0}^{n - 1} \int M^{k}(x,A)\gamma(\rd x)\to \nu(A), \\
 K_{n}(x,E) &= {1 \over 1 + m_{n}} M^{n}(x,E) \sim{u(x)\over n(1+m)}.
\end{align*}
Thus,  \eqref{nex} gives the stated asymptotics for the survival probability, and the stated weak convergence follows from the next corollary of \eqref{tex}:
\[
\rE_x \Big(e^{-\int {w(y)\over n}Z_n(\rd y)}
|Z_n>0\Big)={1-\int (1-e^{-{w(y)\over n}})\kappa_n(\rd y) \over 1 +  m_n\int (1-e^{-{w(y)\over n}})\gamma_n(\rd y)}\to{1\over 1+I_w},
\]
 as $n\to\infty$, where $I_w=(1+m)\beta^{-1}\int w(y)\nu(\rd y)$.
\end{proof}

\begin{proposition}\label{p3}  Assume \eqref{cond}, let   $\rho>1$, and put $c=\beta (\rho-1)/(1+m)$. Then  as $n\to\infty$,
\[\rP_{x}(Z_{n} > 0) \to c u(x),\]
Moreover, for any measurable function $w:E\to (-\infty,\infty)$ with 
$\int w(y)\nu(\rd y)\in(0,\infty)$,
and for any $x\ge0$, we have
\begin{align*}
\rP_{x}\Big({\int w(y)Z_n(\rd y)\over \int w(y)\nu(\rd y)}>\rho^nx | Z_{n} > 0\Big)\to e^{-xc}.
\end{align*}
\end{proposition}
%
%
\begin{proof}  
%
%
From \eqref{mn}, \eqref{KsE}, and \eqref{rad} we see that
\begin{align*}
\sum_{n=1}^\infty m_n s^{n-1}={m(1+f(s))\over (1-mf(s))(1-s)}.
\end{align*}
Thus, Lemma \ref{Fe} with $c_n=R^nm_n$, $a(s)=mf(Rs)$, and $b(s)={m(1+f(sR))\over 1-sR}$ entails
\[
m_n\sim \rho^n(1+m)\beta^{-1}(\rho-1)^{-1}.
\]
This together with  \eqref{nex} gives the stated formula for the survival probability. The assertion on weak convergence is proved in a similar way as in the critical case above.
\end{proof}

\end{document}